\documentclass[12pt]{amsart}
\usepackage[headings]{fullpage}
\usepackage{amssymb,epic,eepic,epsfig,amsbsy,amsmath,amscd,graphicx}
\usepackage[all]{xy}

\numberwithin{equation}{section}
                        \theoremstyle{plain}
\usepackage{mathrsfs}

\newcommand{\psdraw}[2]
         {\begin{array}{c} \hspace{-1.3mm}
         \raisebox{-4pt}{\psfig{figure=#1.eps,width=#2}}
         \hspace{-1.9mm}\end{array}}

\newtheorem{theorem}{Theorem}[section]
\newtheorem{thm}{Theorem}
\newtheorem{lemma}[theorem]{Lemma}

\newtheorem{proposition}[theorem]{Proposition}
\newtheorem{conjecture}{Conjecture}

\theoremstyle{definition}

\newcommand\no[1]{}
\newcommand{\eqM}{\overset{M}{=}}

\newcommand{\lcr}{\raisebox{-5pt}{\mbox{}\hspace{1pt}
                  \epsfig{file=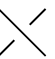}\hspace{1pt}\mbox{}}}
\newcommand{\ift}{\raisebox{-5pt}{\mbox{}\hspace{1pt}
                  \epsfig{file=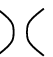}\hspace{1pt}\mbox{}}}
\newcommand{\zer}{\raisebox{-5pt}{\mbox{}\hspace{1pt}
                  \epsfig{file=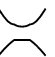}\hspace{1pt}\mbox{}}}

\def\BC{\mathbb C}

\def\BZ{\mathbb Z}
\def\BR{\mathbb R}
\def\BT{\mathbb T}
\def\BJ{\mathbb J}

\def\CA{\mathcal A}

\def\CL{\mathcal L}
\def\CM{\mathcal M}

\def\CP{\mathcal P}
\def\CR{\mathcal R}
\def\CS{\mathcal S}
\def\CT{\mathcal T}

\def\fp{\mathfrak p}
\def\ft{\mathfrak t}

\def\la{\langle}
\def\ra{\rangle}

\def\ve{\varepsilon}
\def\be { \begin{equation} }
\def\ee { \end{equation} }

\begin{document}

\title[AJ conjecture]{On the AJ conjecture for cables of twist knots}

\author[Anh T. Tran]{Anh T. Tran}
\address{Department of Mathematical Sciences, The University of Texas at Dallas,  Richardson TX 75080, USA}
\email{att140830@utdallas.edu}

\begin{abstract}
We study the AJ conjecture that relates the A-polynomial and the colored Jones polynomial of a knot in $S^3$. We confirm the AJ conjecture for $(r,2)$-cables of the $m$-twist knot, for all odd integers $r$ satisfying $\begin{cases} (r+8)(r-8m)>0 &\mbox{if } m> 0, \\ 
r(r+8m-4)>0 &\mbox{if } m<0. \end{cases} $
\end{abstract}

\thanks{2010 {\em Mathematics Classification:} Primary 57N10. Secondary 57M25.}

\maketitle

\section{Introduction}

\subsection{The colored Jones function} For a knot $K$ in $S^3$ and a positive integer $n$, let $J_K(n) \in \BZ[t^{\pm 1}]$ denote the $n^{\text{th}}$ colored Jones polynomial of $K$ with framing 0. The polynomial $J_K(n)$ is the quantum link invariant, as defined by Reshetikhin and Turaev \cite{RT}, associated to the Lie algebra $sl_2(\BC)$, with the color $n$ standing for the simple $sl_2(\BC)$-representation of dimension $n$. Here we use the functorial normalization, i.e. the one for which the colored Jones polynomial of the unknot $U$ is
$$J_U(n)=[n] := \frac{t^{2n}- t^{-2n}}{t^2 -t^{-2}}.$$ 

It is known that $J_K(1)=1$ and $J_K(2)$ is the ordinary Jones polynomial \cite{Jones}. The colored Jones polynomials of higher colors are more or less the ordinary Jones polynomials of parallels of the knot. The color $n$ can be assumed to take negative integer values by setting $J_K(-n) = - J_K(n)$ and $J_K(0)=0$.

For a fixed knot $K$, Garoufalidis and Le \cite{GL} proved that the colored Jones function $J_K: \BZ \to \BZ[t^{\pm 1}]$ satisfies a non-trivial linear recurrence relation of the form $$\sum_{i=0}^da_i(t,t^{2n})J_K(n+i)=0,$$ 
where $a_i(u,v) \in \BC[u,v]$ are polynomials with greatest common divisor 1. 

\subsection{Recurrence relations and $q$-holonomicity} Let $\CR:=\BC[t^{\pm 1}]$. Consider a function $f: \BZ \to \CR$, and define the linear operators $L$ and $M$ acting on such functions by
$$(Lf)(n) := f(n+1), \qquad (Mf )(n) := t^{2n} f(n).$$
It is easy to see that $LM = t^2 ML$. The inverse operators $L^{-1}, M^{-1}$ are well-defined. We can consider $L,M$ as elements of the quantum torus
$$ \mathcal T := \mathbb \CR\langle L^{\pm1}, M^{\pm 1} \rangle/ (LM - t^2 ML),$$
which is a non-commutative ring.

The recurrence ideal of $f$ is the left ideal $\CA_f$ in $\CT$ that annihilates $f$:
$$\CA_f:=\{P \in \CT \mid Pf=0\}.$$
We say that $f$ is $q$-holonomic, or $f$ satisfies a non-trivial linear recurrence relation, if $\CA_f \not= 0$. For example, for a fixed knot $K$ the colored Jones function $J_K$ is $q$-holonomic.

\subsection{The recurrence polynomial} Suppose $f: \BZ \to \CR$ is a $q$-holonomic function. Then $\CA_f$ is a non-zero left ideal of $\CT$. The ring $\mathcal T$ is not a principal left ideal domain. However, it can be embeded into a principal left ideal domain as follows. Let $\CR(M)$ be the fractional field of the polynomial ring $\CR[M]$. Let $\tilde{\CT}$ be the  set of all Laurent polynomials in the variable $L$ with coefficients in $\CR(M)$:
$$\tilde{\CT} =\left\{\sum_{i \in \BZ} a_i(M)L^i \mid a_i(M) \in \CR(M),~a_i=0\text{~almost always}\right\},$$
and define the product in $\tilde{\CT}$ by $a(M)L^{k} \cdot b(M)L^{l}=a(M)b(t^{2k}M)L^{k+l}$.  

Then it is known that $\tilde{\CT}$ is a principal left ideal domain, and $\CT$ embeds as a subring of $\tilde{\CT}$, c.f. \cite{Ga04}. The ideal extension $\tilde\CA_f:=\tilde\CT\mathcal A_f$ of $\CA_f$ in $\tilde{\CT}$ is generated by a polynomial 
$$\alpha_f(t,M,L) = \sum_{i=0}^{d} \alpha_{f,i}(t,M) \, L^i,$$
where the degree in $L$ is assumed to be minimal and all the
coefficients $\alpha_{f,i}(t,M)\in \BC[t^{\pm1},M]$ are assumed to
be co-prime. The polynomial $\alpha_f$ is defined up to a polynomial in $\mathbb C[t^{\pm 1},M]$. We call $\alpha_f$ the recurrence polynomial of  $f$.

When $f$ is the colored Jones function $J_K$ of a knot $K$, we let $\CA_K$ and $\alpha_K$ denote the recurrence ideal $\CA_{J_K}$ and the recurrence polynomial $\alpha_{J_K}$ of $J_K$ respectively. We also say that $\CA_K$ and $\alpha_{K}$ are respectively the recurrence ideal and the recurrence polynomial of the knot $K$. Since $J_K(n) \in \BZ[t^{\pm 1}]$, we can assume that $\alpha_K(t,M,L)=\sum_{i=0}^d \alpha_{K,i}(t,M)L^i$
where all the coefficients $\alpha_{K,i} \in \BZ[t^{\pm 1}, M]$ are co-prime.

\subsection{The AJ conjecture}

The colored Jones polynomials are powerful invariants of knots, but little is known about their relationship with classical topology invariants like the fundamental group. Motivated by the theory of noncommutative A-ideals of Frohman, Gelca and Lofaro \cite{FGL, Ge} and the theory of $q$-holonomicity of quantum invariants of Garoufalidis and Le \cite{GL}, Garoufalidis \cite{Ga04} formulated the following conjecture that relates the A-polynomial and the colored Jones polynomial of a knot in $S^3$.

\begin{conjecture}
\label{c1}
{\bf (AJ conjecture)} For every knot $K$, $\alpha_K |_{t=-1}$ is equal to the $A$-polynomial, up to a factor depending on $M$ only.
\end{conjecture}

The A-polynomial of a knot was introduced by Cooper et al. \cite{CCGLS}; it describes the $SL_2(\BC)$-character variety of the knot complement as viewed from the boundary torus. The A-polynomial carries important information about the topology of the knot. For example, the sides of its Newton polygon give rise to incompressible surfaces in the knot complement \cite{CCGLS}. Here in the definition of the $A$-polynomial, we also allow the factor $L-1$ coming from the abelian character variety of the knot. Hence the A-polynomial in this paper is equal to $L-1$ times the A-polynomial defined in \cite{CCGLS}.

The AJ conjecture has been verified for the trefoil and figure eight knots  \cite{Ga04}, all torus knots  \cite{Hi, Tr}, some classes of two-bridge knots and pretzel knots \cite{Le06, LTaj}, the knot $7_4$  \cite{GK}, and most cabled knots over torus knots and the figure eight knot \cite{RZ, Ru, Traj8}. 

\subsection{Main result} Suppose $K$ is a knot with framing 0, and $r,s$ are two integers with $c$ their greatest common divisor. The $(r,s)$-cable $K^{(r,s)}$ of $K$ is the
link consisting of $c$ parallel copies of the $(\frac{r}{c},\frac{s}{c})$-curve on the torus boundary of a tubular neighborhood of $K$. Here an $(\frac{r}{c},\frac{s}{c})$-curve is a curve that is homologically
equal to $\frac{r}{c}$ times the meridian and $\frac{s}{c}$ times the longitude on the torus boundary.
The cable $K^{(r,s)}$ inherits an orientation from $K$, and we assume that each component of $K^{(r,s)}$ has framing 0. Note that if $r$ and $s$ are co-prime, then $K^{(r,s)}$ is again a knot. 

Consider the $m$-twist knot $K_m$ in Figure 1, where $m$ denotes the number of  full twists;  positive (resp. negative) numbers correspond to right-handed (resp. left-handed) twists. Note that $K_{-1}$ is the trefoil knot and $K_1$ is the figure eight knot.

\begin{figure}[htpb]
$$\psdraw{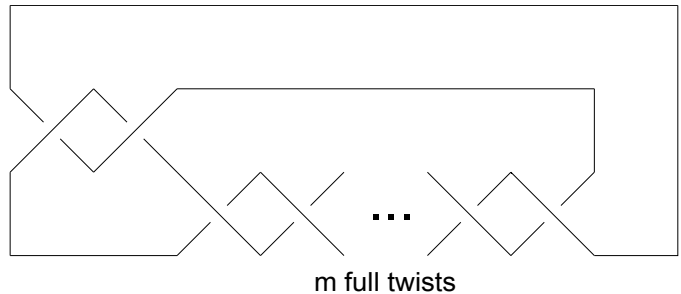}{3.5in} $$
\caption{The $m$-twist knot $K_m$, $m \in \BZ$.}
\label{twist_knot}
\end{figure}
In \cite{RZ, Ru, Traj8} the proof of the AJ conjecture for most cabled knots over torus knots and the figure eight knot is based on explicit formulas for the colored Jones polynomial. In this paper, following the geometric method in \cite{Le06, LTaj} we use skein theory to study the AJ conjecture for cables of twist knots. In particular, we will show the following.

\begin{thm}
\label{main}
The AJ conjecture holds true for $(r,2)$-cables of the $m$-twist knot, for all odd integers $r$ satisfying $\begin{cases} (r+8)(r-8m)>0 &\mbox{if } m> 0, \\ 
r(r+8m-4)>0 &\mbox{if } m<0. \end{cases} $
\end{thm}

\subsection{Plan of the paper} In Section \ref{char_skein} we review character varieties and skein modules. In Section \ref{cable} we prove some properties of the colored Jones polynomial and the recurrence polynomial of cables of a knot. We give a proof of Theorem \ref{main} in Section \ref{twist}.

\subsection{Acknowledgment} I would like to thank Thang T.Q. Le for helpful discussions. 

\section{Character varieties and skein modules}

\label{char_skein}

\subsection{Character varieties and the A-polynomial}

\subsubsection{Character varieties}

The set of representations of a finitely generated group $G$
into $SL_2(\BC)$ is a complex algebraic set, on which
$SL_2(\BC)$ acts by conjugation. The set-theoretic
quotient of the representation space by that action does not
have good topological properties, because two representations with
the same character may belong to different orbits of that action. A better
quotient, the algebro-geometric quotient denoted by $\chi(G)$
(see \cite{CS,LM}), has the structure of an algebraic
set. There is a bijection between $\chi(G)$ and the set of all
characters of representations of $G$ into $SL_2(\BC)$, hence
$\chi(G)$ is also called the character variety of $G$.
For a manifold $Y$ we use $\chi(Y)$ to denote $\chi(\pi_1(Y))$.

Suppose  $G=\BZ^2$, the free abelian group with 2 generators.
Every pair  of generators $\mu,\lambda$ will define an isomorphism
between $\chi(G)$ and $(\BC^*)^2/\tau$, where $(\BC^*)^2$ is the
set of non-zero complex pairs $(M,L)$ and $\tau: (\BC^*)^2 \to (\BC^*)^2$ is the involution
defined by $\tau(M,L):=(M^{-1},L^{-1})$, as follows. Every representation is
conjugate to an upper diagonal one, with $M$ and $L$ being the
upper left entry of $\mu$ and $\lambda$ respectively. The
isomorphism does not change if we replace $(\mu,\lambda)$ with
$(\mu^{-1},\lambda^{-1})$.

\subsubsection{The $A$-polynomial}

\label{Zariski}

Let $K$ be a knot in $S^3$ and $X$ its complement. The boundary of
$X$ is a torus whose fundamental group is free abelian of rank
2. An orientation of $K$ will define a unique pair of an
oriented meridian  and an oriented longitude such that the linking
number between the longitude and the knot is 0. The pair provides
an identification of $\chi(\partial X)$ and $(\BC^*)^2/\tau$
which actually does not depend on the orientation of $K$.

The inclusion $\partial X \hookrightarrow X$ induces the restriction
map
$$\rho : \chi(X) \longmapsto \chi(\partial X)\equiv (\BC^*)^2/\tau.$$
 Let $Z$ be the image of
$\rho$ and  $\hat Z \subset (\BC^*)^2$ the lift of $Z$ under the
projection $(\BC^*)^2 \to (\BC^*)^2/\tau$. The Zariski closure of
$\hat Z\subset (\BC^*)^2 \subset \BC^2$ in $\BC^2$ is an algebraic set
consisting of components of dimension 0 or 1. The union of all the
one-dimensional components is defined by a single polynomial $A_K\in
\BZ[M,L]$, whose coefficients are co-prime. Note that $A_K$ is
defined up to $\pm 1$. We call $A_K$ the $A$-polynomial of $K$.
By definition, $A_K$ does not have repeated factors. It is known that $A_K$ is always divisible by
$L-1$. The $A$-polynomial here is actually equal to $L-1$ times the $A$-polynomial defined in \cite{CCGLS}.

\subsubsection{The $B$-polynomial} 

\label{B}

For a complex algebraic set $V$, let $\BC[V]$ denote the ring
of regular functions on $V$.  For example, $\BC[(\BC^*)^2/\tau]=
\ft^\sigma$, the $\sigma$-invariant subspace of  $\ft:=\BC[L^{\pm
1},M^{\pm 1}]$, where $\sigma(M^{k}L^{l})= M^{-k}L^{-l}$.

The map $\rho$ in the previous subsection induces an algebra
homomorphism
$$\theta: \BC[\chi(\partial X)]  \equiv \ft^\sigma \longrightarrow
\BC[\chi(X)].$$
We call the kernel $\fp$ of $\theta$ the classical
peripheral ideal; it is an ideal of $\ft^\sigma$. 

The ring $\BC[\chi(X)]$ has a $\ft^{\sigma}$-module structure via the map $\theta$: For $f \in \ft^{\sigma}$ and $g \in \BC[\chi(X)]$, the action $f \cdot g$ is defined to be $\theta(f)g \in \BC[\chi(X)]$. Since  $\BC[M^{\pm 1}]^{\sigma}$ is a subring of $\ft^{\sigma}$, $\BC[\chi(X)]$ also has a $\BC[M^{\pm 1}]^{\sigma}$-module structure. Extending the map $\theta: \ft^{\sigma} \to \BC[\chi(X)]$ from the ground ring ${\BC[M^{\pm 1}]}^{\sigma}$ to $\BC(M)$, the fractional field of $\BC[M]$, we get
$$
\left( \bar{\ft} \overset{\bar{\theta}}{\longrightarrow} \overline{\BC[\chi(X)]} \right) :=\left( \ft^\sigma \overset{\theta}{\longrightarrow} \BC[\chi(X)] \right) \otimes _{\BC[M^{\pm 1}]^{\sigma}} \BC(M).
$$

The ring $\bar \ft=\BC(M)[L^{\pm1}]$ is a principal ideal domain. The ideal $\bar{\fp}:= \ker \bar{\theta} \subset \bar{\ft}$ is thus generated by a single polynomial
$B_K\in \BZ[M,L]$ which has co-prime coefficients and is defined
up to a factor $\pm M^k$ with $k\in \BZ$. Again $B_K$ can be
chosen to have integer coefficients because everything can be
defined over $\BZ$. We call $B_K$ the $B$-polynomial of $K$. In \cite[Cor. 2.3]{LTaj} the following is shown.

\begin{proposition}
\label{A_B}
The polynomials $A_K$ and $B_K$ are equal, up to a factor in $\BZ[M]$.
\end{proposition} 

\subsection{Skein modules and the colored Jones polynomial} The theory of the  Kauffman bracket skein module (KBSM) was introduced by Przytycki \cite{Pr} and Turaev \cite{Tu} as a generalization of the Kauffman bracket \cite{Ka} in $S^3$  to an arbitrary
3-manifold. The KBSM of a knot complement contains a lot of information about its colored Jones polynomial.

\subsubsection{Skein modules} Recall that $\CR=\BC[t^{\pm1}]$.
 A framed link in an oriented $3$-manifold $Y$ is a disjoint union of embedded circles, equipped with a
 non-zero normal vector field. Framed links are considered up to isotopy. Let $\mathcal{L}$ be the set of isotopy
classes of framed links in the manifold $Y$, including the empty
link. Consider the free $\CR$-module with basis $\mathcal{L}$, and
factor it by the smallest submodule containing all expressions of
the form $\lcr-t\zer-t^{-1}\ift$ and
$\bigcirc+(t^2+t^{-2}) \emptyset$, where the links in each
expression are identical except in a ball in which they look like
depicted. This quotient is denoted by $\CS(Y)$ and is called the
Kauffman bracket skein module, or just skein module, of $Y$.

For an oriented surface $\Sigma$ we define $\CS(\Sigma):= \CS(Y)$, where $ Y= \Sigma \times [0,1]$ is the cylinder over $\Sigma$. The skein module
$\CS(\Sigma)$ has an algebra structure induced by the operation
of gluing one cylinder on top of the other. The operation of
gluing the cylinder over $\partial Y$ to $Y$ induces a
$\CS(\partial Y)$-left module structure on $\CS(Y)$.

\subsubsection{The colored Jones polynomial} When $Y=S^3$, the skein
module $\CS(Y)$ is free over $\CR$ of rank 1, and is spanned by
the empty link. Thus if $\ell$ is a framed link in $S^3$, then
its value in $\CS(S^3)$ is $\langle \ell
\rangle$ times the empty link, where $\langle \ell \rangle \in
\CR$ is the Kauffman bracket of $\ell$ \cite{Ka} which is the Jones polynomial of the
framed link $\ell$ in a suitable normalization.

Let $\{S_n(z)\}_{n \in \BZ}$ be the Chebychev polynomials defined by $S_0(z)=1$, $S_1(z)=z$ and $S_{n+1}(z)=zS_n(z)-S_{n-1}(z)$ for all $n \in \BZ$. For a framed knot $K$ in $S^3$ and an integer $n > 0$, we define the $n^{\text{th}}$
power $K ^n$ as the link consisting of $n$ parallel copies of $K$.
Using these powers of a knot, $S_n(K)$ is defined as an element
of $\CS(S^3)$. We define the colored Jones polynomial $J_K(n)$ by the equation
$$ J_K(n+1):=(-1)^{n} \times\langle S_n(K) \rangle. $$
The $(-1)^n$ sign is added so that for the unknot $U$, $J_U(n) = [n].$ Then
$J_K(1)=1$ and $J_K(2)= - \langle K \rangle$. We extend this
definition for all integers $n$ by $J_K(-n)= -J_K(n)$ and
$J_K(0)=0$. In the framework of quantum invariants, $J_K(n)$ is
the $sl_2(\BC)$-quantum invariant of $K$ colored by the $n$-dimensional
simple representation of $sl_2(\BC)$.

\subsubsection{The skein module of the torus} Let $\BT^2$ be the torus with a fixed pair $(\mu, \lambda)$ of simple closed curves intersecting at exactly one point.
For co-prime integers $k$ and $l$, let $\lambda_{k,l}$ be a simple closed curve  on the torus homologically equal to $k\mu+ l\lambda$. It is not difficult to show that
the skein algebra $\CS(\BT^2)$ of the torus is generated, as an $\CR$-algebra, by all $\lambda_{k,l}$'s. In fact,
Bullock and Przytycki \cite{BP} showed that  $\CS(\BT^2)$ is
generated over $\CR$ by 3 elements $\mu,\lambda$ and $\lambda_{1,1}$, subject to some explicit relations.

Recall that $ \CT= \CR\la M^{\pm 1}, L^{\pm1} \ra/(LM - t^2 ML)$ is the quantum torus. Let $\sigma: \CT \to \CT$ be the involution defined by $\sigma(M^{k} L^{l}) := M^{-k} L^{-l}$.
Frohman and Gelca \cite{FG} showed that there is an algebra isomorphism $\Upsilon: \CS(\BT^2)\to \CT^{\sigma}$ given by
$$ \Upsilon(\lambda_{k,l}) :=  (-1)^{k+l} t^{kl} (M^{k}L^{l} +M^{-k}L^{-l}).$$
The fact that $\CS(\BT^2)$ and $\CT^{\sigma}$ are isomorphic algebras was also proved by Sallenave  \cite{Sa}.
\label{sec_torus}

\subsection{Character varieties and skein modules}

\subsubsection{Skein modules as quantizations of character varieties}

Let $\ve$ be the map reducing $t=-1$. An important result of Bullock, Przytycki and Sikora \cite{Bu,PS} in the theory of skein modules is that
$\varepsilon (\CS(Y))$ has a natural algebra structure and, when factored by its nilradical, is canonically isomorphic to the character ring $\BC[\chi(Y)]$. The product of two links in $\varepsilon (\CS(Y))$ is their disjoint union, which is well-defined when $t=-1$. The isomorphism between $\varepsilon (\CS(Y))/\sqrt{0}$ and $\BC[\chi(Y)]$ is given by $K(r)= -\text{tr}\, r(K)$,
where $K$ is a knot in $Y$ representing an
element of $\pi_1(Y)$ and $r:\pi_1(Y) \to
SL_2(\BC)$ is a representation.

In many cases the nilradical of $\varepsilon (\CS(Y))$ is trivial, and hence $\varepsilon (\CS(Y))$ is exactly equal to the character ring $\BC[\chi(Y)]$. For example, this is the case when $Y$ is a torus, or when $Y$ is the
complement of a two-bridge knot/link \cite{Le06, PS, LTskein}, or when $Y$ is the complement of the $(-2,3,2n + 1)$-pretzel knot \cite{LTaj}.

\no{
Suppose that $\varepsilon (\CS(X))=\BC[\chi(X)]$ where $X$ is the complement of a knot $K$ in $S^3$. Then we have the following commutative diagram with exact rows.

\begin{equation} 
\label{diagram1}
\begin{CD}  
0 @ >>> \CP @ >>> \CT^{\sigma}   @>\Theta >> \CS(X)  \\
@.  @ VVV @V \varepsilon VV  @V \varepsilon VV \\
0 @ >>> \fp  @ >>>  \ft^{\sigma}     @>\theta >> \BC[\chi(X)] 
\end{CD}
\end{equation} 
}

\subsubsection{The quantum peripheral ideal} Recall that $X$ is the complement of a knot $K$ in $S^3$. There is a standard choice of a meridian $\mu$  and a longitude $\lambda$ on $\partial X=\BT^2$ such that the linking
number between the longitude and the knot is 0, as in Subsection \ref{Zariski}. We use this pair $(\mu,\lambda)$ and the map $\Upsilon$ in the previous subsection to identify $\CS(\partial X)$ with $\CT^\sigma$.

The operation of
gluing the cylinder over $\partial X$ to $X$ induces a
$\CT^\sigma$-left module structure on $\CS(X)$: For $\ell \in \CT^\sigma=\CS(\partial X)$ and $\ell' \in \CS(X)$, the action $\ell \cdot \ell' \in \CS(X)$ is the disjoint union of $\ell$ and $\ell'$. In general $\CS(X)$ does not have an algebra structure, but it has the identity element--the empty link $\emptyset$.
The map
$$ \Theta: \CS(\partial X) \equiv \CT^\sigma \to \CS(X), \quad \Theta(\ell) := \ell \cdot \emptyset$$
can be considered as a quantum analog of the map $\theta: \ft^{\sigma} \to \BC[\chi(X)]$ defined in Subsection \ref{B}. It is $\CT^\sigma$-linear and its kernel $\CP:=\ker \Theta$ is called the quantum peripheral ideal, first introduced in \cite{FGL}. In \cite{FGL, Ge}, it was proved that every element in $\CP$ gives rise to a recurrence relation for the colored Jones polynomial. In \cite{Ga08} the following stronger result was shown (see also \cite[Cor. 1.2]{LTaj} for an alternative proof).

\begin{proposition}
The quantum peripheral ideal is in the recurrence ideal, i.e. $\CP \subset \CA_K$.
\label{PA}
\end{proposition}

\subsubsection{Localization} 

Let $D:=\CR [M^{\pm 1}]$ and $\bar{D}$ be its localization at $(1+t)$, i.e. 
$$\bar{D}:=\left\{\frac{f}{g} \mid f,g \in D, \, g \not \in (1+t)   D     
\right\}
,$$
which is flat over $D$. The ring $D=\CR [M^{\pm 1}]$ is flat over $D^\sigma=\CR [M^{\pm 1}]^\sigma$, where $\sigma(M)=M^{-1}$, since it is free over $\CR [M^{\pm 1}]^\sigma$. Hence $\bar{D}$ is flat over $D^\sigma$.

The skein module $\CS(X)$ has a $\CT^{\sigma}$-module structure, hence a $D^\sigma$-module structure since $D^\sigma$ is a subring of $\CT^{\sigma}$. Extending the map $\Theta: \CT^{\sigma} \to \CS(X)$ in the previous subsection from the ground ring $D^\sigma$ to $\bar{D}$ we get
$$
\left( \bar{\CT} \overset{\bar{\Theta}}{\longrightarrow} \overline{\CS(X)} \right) :=\left( \CT^\sigma \overset{\Theta}{\longrightarrow} \CS(X) \right) \otimes _{D^\sigma} \bar{D}.
$$

In \cite{LTaj}, the $\bar{D}$-module $\overline{\CS(X)}$ is called the localized skein module of the knot complement $X$. The ring $\bar{\CT}$ can be explicitly described as 
$$\bar{\CT} =\left\{\sum_{i \in \BZ} a_i(M)L^i \mid a_i(M) \in \bar{D},~a_i=0\text{~almost always}\right\},$$
with commutation rule $a(M)L^{k} \cdot b(M)L^{l}=a(M)b(t^{2k}M)L^{k+l}$.

 Let $\bar{\CP} := \ker \bar{\Theta} \subset \bar{\CT}$. It can be shown that $\bar{\CP}$ is the ideal extension of $\CP \subset \CT^{\sigma}$ in $\bar{\CT}$. Although $\bar{\CT}$ is not a principal left ideal domain, its ideals (and in particular $\bar{\CP}$) has nice descriptions which are very useful in the approach to the AJ conjecture in \cite{LTaj}.

\section{Recurrence polynomials of cable knots} 

\label{cable}

In this section we prove some properties of the colored Jones polynomial and the recurrence polynomial of $(r,2)$-cables of a knot $K$, where $r$ is an odd integer. 

For $n>0$, by \cite[Sec. 2.1]{LTvol} we have
\begin{equation}
\label{cables}
J_{K^{(r,2)}}(n) = t^{-2r(n^2-1)}\sum_{i=1}^n (-1)^{r(n-i)} t^{2ri(i-1)} J_K(2i-1).
\end{equation} 

Let $\BJ_K(n):=J_K(2n+1)$. Then
\begin{equation}
\label{f}
M^r(L+t^{-2r}M^{-2r})J_{K^{(r,2)}} = \BJ_K,
\end{equation} 
see \cite[Sec. 6.1]{RZ} or \cite[Lem. 2.1]{Traj8}.

For a Laurent polynomial $f(t) \in \CR$, let $d_+[f]$ and $d_-[f]$ be respectively the maximal and minimal degree of $f$ in $t$. 

For a knot diagram $D$, let $k_+(D)$ (resp. $k_-(D)$) be the number of positive (resp. negative) crossings of $D$, and $s_+(D)$ (resp. $s_-(D)$) the number of circles obtained by positively (resp. negatively) smoothing all the crossings of $D$. Let $k(D)=k_+(D) + k_-(D)$ be the number of crossings and $w(D)=k_+(D) - k_-(D)$ the writhe of $D$.

We will use the following result. For a definition of adequate knots, see \cite[Chap. 5]{Li}.

\begin{lemma} 
\label{degLe}
\cite[Prop. 2.1]{Le06}
Suppose $K$ is a (zero-framed) knot with an adequate diagram $D$. Then, for $n>0$, we have
\begin{eqnarray*}
d_+[J_{K}(n)] &=&  k(D)(n-1)^2+2(n-1)s_+(D)-w(D)(n^2-1),\\
d_-[J_{K}(n)] &=& -k(D)(n-1)^2-2(n-1)s_-(D)-w(D)(n^2-1),
\end{eqnarray*}
\end{lemma}

\begin{lemma}
\label{degree}
Suppose $K$ is a non-trivial knot with an adequate diagram $D$. For $n>0$, 
\begin{eqnarray*}
d_+[J_{K^{(r,2)}}(n)] &=& -2rn^2+2r \quad \mbox{if } r < -4k_-(D), \\
d_-[J_{K^{(r,2)}}(n)] &=& -2rn^2+2r \quad \mbox{if } r > 4k_+(D).
\end{eqnarray*}
\end{lemma}

\begin{proof} We prove the formula for $d_+[J_{K^{(r,2)}}(n)]$. The one for $d_-[J_{K^{(r,2)}}(n)]$ is proved similarly. By equation \eqref{cables},
\begin{equation}
\label{d+}
d_+[J_{K^{(r,2)}}(n)] = -2r(n^2-1) +  d_+ \big[ \sum_{i=1}^n (-1)^{r(n-i)} t^{2ri(i-1)} J_K(2i-1) \big].
\end{equation}

Consider the quadratic polynomial $f(x)$ defined by
$$
f(x) = \big( 2r+8k_-(D) \big) x^2 - \big( 2r+8k_-(D) + 4k(D) -4s_+(D) \big) x+4k(D)-4s_+(D)
$$
where $x \in \BR$. From Lemma \ref{degLe} we have $$
d_+[t^{2ri(i-1)} J_K(2i-1)] = 2ri(i-1)+ d_+[J_{K}(2i-1)] = f(i).$$

Suppose $r < -4k_-(D)$. The quadratic polynomial $f(x)$ is concave down and attains its maximum on the real line at $x=x_0:=\frac{1}{2}+\frac{k(D)-s_+(D)}{r+4k_-(D)}$. By \cite[Lem. 5.6]{Li}, we have $s_+(D) + s_-(D) \le k(D) +2$. Since $K$ is non-trivial, $s_-(D) \ge 2$ and hence we must have $s_+(D) \le k(D)$. This implies that $x_0 \le \frac{1}{2}$ and so $f(x)$ is a strictly decreasing function on the interval $[1,n]$. Hence, for $i=2,3,\cdots,n$, we have $f(i)<f(1)=d_+[J_K(1)]=0$. Then  
$$d_+ \big[ \sum_{i=1}^n (-1)^{r(n-i)} t^{2ri(i-1)} J_K(2i-1) \big]=0.$$
The lemma now follows from equation \eqref{d+}.
\end{proof}

\begin{proposition}
Suppose $K$ is a non-trivial knot with an adequate diagram $D$. If $r$ is an odd integer with $r < -4k_-(D)$ or $r > 4k_+(D)$, then $\alpha_{K^{(r,2)}}=\alpha_{\BJ_K}M^r(L+t^{-2r}M^{-2r})$. 
\label{divisible}
\end{proposition}

\begin{proof}
We first claim that $\alpha_{K^{(r,2)}}$ is left divisible by $M^r(L+t^{-2r}M^{-2r})$. Indeed, write $$\alpha_{K^{(r,2)}}=QM^r(L+t^{-2r}M^{-2r})+R$$ where $Q \in \CR[M^{\pm 1}][L]$ and $R \in \CR[M^{\pm 1}]$. 

From equation \eqref{f} we have
\begin{eqnarray}
\label{div}
0 &=& \alpha_{K^{(r,2)}} J_{K^{(r,2)}} \nonumber\\
  &=& QM^r(L+t^{-2r}M^{-2r})J_{K^{(r,2)}}+RJ_{K^{(r,2)}}\nonumber\\
  &=& Q \BJ_K + RJ_{K^{(r,2)}}.
\end{eqnarray}

Assume that $R \not=0$. Write $Q=\sum_{i=0}^d a_i(M) L^i$ where $a_i(M) \in \CR[M^{\pm 1}]$. If $r > 4k_+(D)$ then, for $n$ big enough, by Lemma \ref{degree},  $$d_-[RJ_{K^{(r,2)}}(n)] = d_-[J_{K^{(r,2)}}(n)] + O(n)= -2rn^2 + O(n).$$ 
Similarly, if $a_i \not= 0$ then by Lemma \ref{degLe}, 
$$d_-[(a_i L^i \BJ_K)(n)]=d_-[J_K(2n+2i+1)]+O(n) \ge -8k_+ n^2+O(n).$$ 
It follows that for $n$ big enough we have
$$d_-[RJ_{K^{(r,2)}}(n)] < \min_{0 \le i \le d} \{ d_-[(a_i L^i \BJ_K)(n)] \} \le d_-[Q \BJ_K(n)].$$ This contradicts equation \eqref{div}. 

If $r < -4k_-(D)$ then, by similar arguments as above we have 
$$d_+[RJ_{K^{(r,2)}}(n)] > \max_{0 \le i \le d} \{ d_+[(a_i L^i \BJ_K)(n)] \} \ge d_+[Q \BJ_K(n)],$$
for $n$ big enough. This also contradicts equation \eqref{div}. 

Hence $R=0$, which means $\alpha_{K^{(r,2)}}$ is left divisible by $M^r(L+t^{-2r}M^{-2r})$. Then, since $M^r(L+t^{-2r}M^{-2r})J_{K^{(r,2)}} = \BJ_K$,
it is easy to see that $\alpha_{K^{(r,2)}}=\alpha_{\BJ_K}M^r(L+t^{-2r}M^{-2r})$.
\end{proof}

\section{Proof of Theorem \ref{main}}

\label{twist}

Consider the $m$-twist knot $K_m$ in Figure \ref{twist_knot}. It has a reduced alternating (and hence adequate) diagram $D$ with 
\begin{eqnarray*}
k_+(D) &=& \begin{cases} 2m & \mbox{if } m>0, \\
1-2m &\mbox{if } m<0, \end{cases}, \qquad
k_-(D) = \begin{cases} 2 & \mbox{if } m>0, \\
0 &\mbox{if } m<0. \end{cases}
\end{eqnarray*}

For non-zero $f,g \in \BC(M)[L]$, we write $f \eqM g$ if the quotient $f/g$ does not depend on $L$.

\subsection{The A-polynomial} A formula for the A-polynomial of cabled knots has recently been given by Ni and Zhang, c.f. \cite{Ru}. In particular, for an odd integer $r$ we have 
\begin{equation}\label{A_cable}A_{K^{(r,2)}}(L,M)=
(L-1) \text{Res}_{\lambda} \left( \frac{A_{K}(\lambda,M^2)}{\lambda-1},\lambda^2-L \right)F(L,M),
\end{equation}
where $\text{Res}_{\lambda}$ denotes the polynomial resultant eliminating the variable $\lambda$ and
$$F(L,M):= \begin{cases} M^{2r}L+1 &\mbox{if } r>0, \\
L+M^{-2r} & \mbox{if } r<0. \end{cases}$$

\begin{lemma}
Suppose $P(L,M) \in \BC[L,M]$ is irreducible and $P(L,M) \not= P((-L,M)$. Let $Q(L,M):=\emph{Res}_{\lambda} (P(\lambda,M), \lambda^2-L).$ Then $Q(L,M)\in \BC[L,M]$ is irreducible and has $L$-degree equal to that of $P(L,M)$.
\label{P_Q}
\end{lemma}

\begin{proof}
Since $P(L,M) \in \BC[L,M]$ is irreducible and $P(L,M) \not= P(-L,M)$, we have
$$
Q(L,M)=P(\sqrt{L},M)P(-\sqrt{L},M).
$$
Then $Q$ has $L$-degree equal to that of $P(L,M)$. It remains to show that $Q(L,M)$ is irreducible in $\BC[L,M]$.

Suppose that $Q(L,M)=Q_1(L,M)Q_2(L,M)$, where $Q_i(L,M) \in \BC[L,M]$. By replacing $L$ by $\lambda^2$, we have 
\begin{equation}
\label{Q1Q2}
Q_1(\lambda^2,M)Q_2(\lambda^2,M)=P(\lambda,M)P(-\lambda,M).
\end{equation}
Consider equation \eqref{Q1Q2} in $\BC[\lambda, M]$. Without loss of generality, we may assume that $P(\lambda,M)$ divides $Q_1(\lambda^2,M)$. Then $P(\lambda,M)P(-\lambda,M)$ also divides $Q_1(\lambda^2,M)$. Equation \eqref{Q1Q2} implies that $Q_2(\lambda^2,M)$ is a constant polynomial. Hence $Q(L,M)$ is irreducible.
\end{proof}

If $P(L,M)=\sum_{i,j} a_{ij} L^iM^j \in \BC[L,M]$ then its Newton polygon is the smallest convex set in the plane containing all integral lattice points $(i,j)$ for which $a_{ij} \not=0$. 

We recall some properties of the A-polynomial of the twist knot $K_m$ from \cite{HS} in the following lemma. Note that $K_m=J(2,-2m)$ in the notation of \cite{HS}.

\begin{lemma} \cite{HS}
Write $$A_{K_m}(L,M)=(L-1)A'_{K_m}(L,M).$$ Then the polynomial $A'_{K_m}(L,M) \in \BZ[L,M]$ has the following properties.

(i) $A'_{K_m}(L,M)$ has $L$-degree equal to $2m$ if $m>0$ and $-(2m+1)$ if $m<0$.

(ii) $A'_{K_m}(L,M)$ is irreducible in $\BC[L,M]$. 

(iii) the Newton polygon of $A'_{K_m}(L,M)$ has the following vertex set
$$\{(1,8m),(m,8m),(0,4m),(2m,4m),(m,0),(2m-1,0)\}$$
if $m>0$, and
$$\{(0,-8m-2),(-m-1,-8m-2),(1,-4m-4),(-2m-2,-4m+2),(-m,0),(-2m-1,0)\}$$
if $m<0$.

(iv) $A'_{K_m}(L,0)=L^{|m|}(L-1)^{|m|-1}.$
\label{A_poly_twist}
\end{lemma}

\begin{proposition} 
Let $$R_{K_m}(L,M):=\emph{Res}_{\lambda} (A'_{K_m}(\lambda, M), \lambda^2-L).$$ Then each of $R_{K_m}(L,M)$ and $R_{K_m}(L,M^2)$ is irreducible in $\BC[L,M]$, and has $L$-degree equal to $2m$ if $m>0$ and $-(2m+1)$ if $m<0$.
\label{R}
\end{proposition}

\begin{proof}
From the Newton polynomial of $A'_{K_m}(L,M)$ in Lemma \ref{A_poly_twist}, we see that $A'_{K_m}(L,M)$ contains at least one monomial $a_{ij}L^iM^j$, where $i$ is an odd integer. This implies that $A'_{K_m}(L,M) \not= A'_{K_m}(-L,M)$. By Lemma \ref{A_poly_twist}, $A'_{K_m}(L,M)$ is irreducible and has $L$-degree equal to $2m$ if $m>0$ and $-(2m+1)$ if $m<0$. Hence, by Lemma \ref{P_Q}, the same conclusion holds true for $R_{K_m}(L,M)=\text{Res}_{\lambda} (A'_{K_m}(\lambda, M), \lambda^2-L)$. Moreover, $$R_{K_m}(L,M)=A'_{K_m}(\sqrt{L}, M)A'_{K_m}(-\sqrt{L}, M).$$

It remains to show that $R_{K_m}(L,M^2)$ is irreducible in $\BC[L,M].$ Otherwise, we must have $R_{K_m}(L,M^2)=S(L,M)S(L,-M)$ for some $S \in \BC[L,M]$. In particular, $$\left( S(L,0) \right)^2=R_{K_m}(L,0)=A'_{K_m}(\sqrt{L},0)A'_{K_m}(-\sqrt{L},0)=- L^{|m|}(L-1)^{|m|-1},$$
since $A_{K_m}(\lambda,0)= \lambda^{|m|}(\lambda-1)^{|m|-1}$ by Lemma \ref{A_poly_twist}. This cannot occur since $|m|$ and $|m|-1$ cannot be not simultaneously even.
\end{proof}

\subsection{The recurrence polynomial} 

Suppose $r$ is an odd integer satisfying $$\begin{cases} (r+8)(r-8m)>0 &\mbox{if } m> 0, \\ 
r(r+8m-4)>0 &\mbox{if } m<0. \end{cases} $$ 
By Proposition \ref{divisible}, we have $\alpha_{K_m^{(r,2)}}=\alpha_{\BJ_{K_m}} (L+t^{-2r}M^{-2r})$.

\begin{lemma}
\label{L2}
For $P(L,M) \in \CT$ we have
$$\left( P(L^2,M)J_K \right)(2n+1) = \left( P(L,t^2M^2)\BJ_K \right)(n).$$
\end{lemma}

\begin{proof}
This is because $(M^k L^{2l}J_K)(2n+1)=((t^2M^2)^kL^{l}\BJ_K)(n)$.
\end{proof}

To determine $\alpha_{\BJ_{K_m}}$ (and hence $\alpha_{K_m^{(r,2)}}$), by Lemma \ref{L2} we need to find an element $P(L,M) \in \CT$ of minimial $L$-degree such that $P(L^2,M)$ annihilates the colored Jones function $J_{K_m}$. Note that this is done for the figure eight knot $K_1$ in \cite{Ru, Traj8} by using explicit formulas for the colored Jones polynomial. We now use skein theory to show the existence of such an element $P$ for all twist knots.

\subsubsection{The action matrix} Recall that $D=\CR [M^{\pm 1}]$ and $\bar{D}$ its localization at $(1+t)$. Let 
$$d=\begin{cases} 2m &\mbox{if } m>0, \\
-(2m+1) & \mbox{if } m<0. \end{cases}$$

For the twist knot $K_m$, we know that $\CS (X)$ is a free $D^{\sigma}$-module with basis $\{y^i: 0 \le i \le d\}$, see \cite{BL, Le06}. This implies that $\overline{\CS (X)}=\CS (X) \otimes_{D^{\sigma}} \bar{D}$  is a free $\bar{D}$-module with basis $\{y^i: 0 \le i \le d\}$. Note that $\overline{\CT}=\CT^{\sigma} \otimes_{D^{\sigma}} \bar{D}$  is a free $\bar{D}$-module with basis $\{L^i: i \in \BZ\}.$

Recall that $\CS(\partial X)=\CT^{\sigma}$ and $\CS(X)$ has a $\CT^{\sigma}$-left module structure via the gluing of the cylinder over $\partial X=\BT^2$ to $X$. This implies that $\overline{\CS(X)}$ has a $\overline{\CT}$-module structure. We study the action of $\overline{\CT}$ on the $\bar{D}$-module $\overline{\CS (X)}.$ Denote the action of $L$ on  $\overline{\CS (X)}$ by a matrix $\CL \in \text{Mat}_{(d+1) \times (d+1)}(\bar{D})$ and let $e_i:=y^i$ for $0 \le i \le d$.  

We have 
$L \cdot e_i=\sum_j \CL_{ij}e_j$ and
$$L^2 \cdot e_i = L \cdot (L \cdot e_i) = L \cdot (\sum_j \CL_{ij}(M)e_j)
=\sum_j \tau(\CL_{ij}(M))(L \cdot e_j),$$
where $\tau (f(M)):=f(t^2M)$ for $f(M) \in \bar{D}$. Hence 
$$L^2 \cdot e_i = \sum_j \tau(\CL)_{ij}\sum_k \CL_{jk}e_k
=\sum_k (\sum_j \tau(\CL)_{ij}\CL_{jk}) e_k = \sum_k [\tau (\CL)\CL]_{ik} e_k,$$
which means that $L^2$ acts on $\overline{\CS (X)}$ as the matrix $\tau (\CL)\CL$. 
By induction, we can show that $L^j$ acts on $\overline{\CS (X)}$ as the matrix 
$\CL^{(j)}:=\tau^{j-1} (\CL)\tau^{j-2} (\CL) \cdots \tau (\CL)\CL.$

\subsubsection{An annihilator of $\BJ_{K_m}$} Let $v_0:=[1,0, \cdots, 0]^{\text{T}}$ be a vector in $\bar{D}^{d+1}$. Recall that $\bar{\Theta}: \bar{\CT} \to \overline{\CS(X)}$ is the map defined by $\bar{\Theta}(\ell)=\ell \cdot \emptyset$, and $\bar{\CP} \subset \bar{\CT}$ its kernel.

\begin{lemma}
Let $P(L,M) = \sum_i a_i(M) L^i \in \bar{\CT}$ where $a_i(M) \in \bar{D}$. Then $P \in \bar{\CP}$ if and only if $\sum_i a_i(M)  \CL^{(i)}v_0=0$.
\label{v_0}
\end{lemma}

\begin{proof}
By definition, $P \in \bar{\CT}$ is in $\bar{\CP}$ if and only if $P \cdot e_0=0$ in $\overline{\CS(X)}$. We have 
$$P \cdot e_0 = (\sum_i a_i L^i) \cdot e_0 = \sum_i a_i \sum_j \CL^{(i)}_{0j} e_j = \sum_j ( \sum_i a_i  \CL^{(i)}_{0j}) e_j.$$
Hence $P \cdot e_0=0$ if and only if $\sum_i a_i  \CL^{(i)}_{0j}=0$ for all $j$. This system of equations can be rewritten as $\sum_i a_i  [\CL^{(i)}_{00}, \cdots, \CL^{(i)}_{0m}]^{\text{T}}=0$, which is equivalent to $\sum_i a_i  \CL^{(i)}v_0=0$. 
\end{proof}

Let $v_j:=\CL^{(2j)}v_0 \in \bar{D}^{d+1}$ for $1 \le j \le d+1$, and $\CM:=[v_0, \cdots v_{d}] \in \text{Mat}_{(d+1) \times (d+1)}(\bar{D}).$ 

\begin{proposition}
\label{lin}
We have $$\det \ve(\CM) \not= 0.$$
\end{proposition}

\begin{proof}

This is equivalent to show that the vectors $\ve(v_0), \cdots, \ve(v_d)$ are linearly independent in $\BC(M)^{d+1}$. Note that $\ve(\CS(X))=\BC[\chi(X)]$ for the twist knot $K_m$. 

Recall the action of $\ft^{\sigma}$ on the $\BC[M^{\pm 1}]^{\sigma}$-module $\BC[\chi(X)]$ in Subsection \ref{B}. This action induces an action of $\bar{\ft}=\BC(M)[L^{\pm 1}]$ on the $\BC(M)$-vector space $\overline{\BC[\chi(X)]}$. Then the matrix of the action of $L$ on $\overline{\BC[\chi(X)]}$ is equal to $\ve(\CL) \in \text{Mat}_{(d+1) \times (d+1)}(\BC(M))$. From the definition of the B-polynomial of $K_m$, we see that $B_{K_m}(\ve(\CL),M)$ is the minimal polynomial of $\ve(\CL)$. Moreover, by Proposition \ref{A_B} we have
$$B_{K_m}(L,M) \eqM A_{K_m}(L,M)=(L-1)A'_{K_m}(L,M).$$

Let $\CL'=\CL^{(2)}$ and $C_{K_m}(\ve(\CL'),M)$ be the minimal polynomial of $\ve(\CL')=\ve(\CL)^2$. Then 
\begin{eqnarray*}
C_{K_m}(L',M) &=& \text{Res}_L (B_{K_m}(L,M), L^2-L') \\
              &\eqM& \text{Res}_L ((L-1)A'_{K_m}(L, M), L^2-L')\\
              &=& (L'-1)R_{K_m}(L',M).
\end{eqnarray*}

By Proposition \ref{R}, $R_{K_m}(L',M)$ has $L'$-degree $d$ and hence $C_{K_m}(L',M)$ has $L'$-degree $d+1$. By Lemma \ref{v_0}, $C_{K_m}(\ve(\CL'),M)$ is also a polynomial in $\BC[\ve(\CL'),M]$ of minimal $\ve(\CL')$-degree that annihiates $\ve(v_0)$. Since $\ve(v_i)=\ve(\CL')^i \ve(v_0)$, it follows that the vectors $\ve(v_0), \cdots, \ve(v_d)$ are linearly independent in $\BC(M)^{d+1}$. Hence $\det \ve(\CM) \not= 0$.
\end{proof}

We are now ready to define an annihilator of $\BJ_{K_m}$. For $0 \le i \le d$ let $\CM_i$ be the matrix obtained by replacing the $(i+1)^{\text{th}}$ column of $\CM$ by the column vector $v_{d+1}$. Let $b_i(M):= \frac{\det \CM_i}{\det \CM}.$ Proposition \ref{lin} implies that $b_i(M) \in \bar{D}$. Let
$$\beta(L,M) := -L^{d+1} + \sum_{i=0}^{d} b_i(t^2M^2) L^{i}  \in \bar{\CT}.$$

\begin{proposition}
\label{beta}
We have $\beta \BJ_{K_m}=0$. Moreover $$\ve(\beta) \eqM (L-1)R_{K_m}(L,M^2).$$
\end{proposition}

\begin{proof}
From the definition of $b_i(M)$ and Cramer's rule we have $\sum_{i=0}^{d} b_i(M)v_i=v_{d+1}$, which means that $( -\CL^{(2d+2)}+\sum_{i=0}^{d} b_i(M) \CL^{(2i)} )v_0=0$. Then, by Lemma \ref{v_0}, $ -L^{2d+2} + \sum_{i=0}^{d} b_i(M) L^{2i}$ is an element in $\bar{\CP}$.

Let $\bar{\CA} := \CA \otimes_{D} \bar{D}$ be the localized recurrence ideal of $K_m$. Proposition \ref{PA} implies that $\bar{\CP} \subset \bar{\CA}$ and hence $\ -L^{2d+2} + \sum_{i=0}^{d} b_i(M) L^{2i} \in \bar{\CA}$.  By Lemma \ref{L2} we have 
\begin{eqnarray*}
\beta \BJ_{K_m}(n) &=& ( -L^{d+1} + \sum_{i=0}^{d} b_i(t^2M^2) L^{i} )\BJ_{K_m}(n)\\
&=& (  -L^{2d+2} + \sum_{i=0}^{d} b_i(M) L^{2i} ) J_{K_m}(2n+1)=0.\\
\end{eqnarray*}
Hence $\beta \BJ_{K_m}=0$. It remains to show that $\ve(\beta) \eqM C_{K_m}(L,M^2)$.

Write $C_{K_m}(L,M)=-c_{d+1}(M)L^{d+1}+\sum_{i=0}^{d} c_i(M)L^i$, where $c_i(M) \in \BC(M)$ for $0 \le i \le d+1$ and $c_{d+1}(M) \not =0$. Since $C_{K_m}(\ve(\CL'),M)$ is a polynomial in $\BC[\ve(\CL'),M]$ that annihiates $\ve(v_0)$, we have 
$$\big( -c_{d+1}(M)\ve(\CL')^{d+1}+\sum_{i=0}^{d} c_i(M) \ve(\CL')^i \big) \ve(v_0)=0.$$
This implies that $\sum_{i=0}^d \frac{c_i(M)}{c_{d+1}(M)} \, \ve(v_i) = \ve(v_{d+1}).$ Since the vectors $\ve(v_0), \cdots, \ve(v_d)$ are linearly independent in $\BC(M)^{d+1}$, Cramer's rule implies that $$\frac{c_i(M)}{c_{d+1}(M)}=\frac{\det \ve(\CM_i)}{\det \ve(\CM)}=\ve(b_i(M)).$$ 
Hence 
\begin{eqnarray*}
\ve(\beta) &=&  -L^{d+1} + \sum_{i=0}^{d} \ve(b_i(t^2M^2)) L^{i} \\
           &=& -L^{d+1} + \sum_{i=0}^{d} \frac{c_i(M^2)}{c_{d+1}(M^2)} L^{i} \\
           &\eqM& C_{K_m}(L,M^2).
\end{eqnarray*}
The proposition follows since $C_{K_m}(L,M^2)=(L-1)R_{K_m}(L,M^2)$.
\end{proof}

\subsection{Completing the proof of Theorem \ref{main}} Since $K_m$ is an alternating knot, Propositions 2.5 and 2.6 in \cite{Traj8} imply that $\ve(\alpha_{\BJ_{K_m}})$ is divisible by $L-1$ and has $L$-degree $>1$. By Proposition \ref{beta}, $\beta$ annihilates $\BJ_{K_m}$ and hence is left divisible by $\alpha_{\BJ_{K_m}}$. We have $\ve(\alpha_{\BJ_{K_m}})$ divides $\ve(\beta) \eqM (L-1)R_{K_m}(L,M^2)$, and hence $\frac{\ve(\alpha_{\BJ_{K_m}})}{L-1}$ divides $R_{K_m}(L,M^2)$ in $\BC(M)[L]$. Since $\frac{\ve(\alpha_{\BJ_{K_m}})}{L-1}$ has $L$-degree $\ge 1$ and $R_{K_m}(L,M^2)$ is irreducible in $\BC[L,M]$, we conclude that $\frac{\ve(\alpha_{\BJ_{K_m}})}{L-1} \eqM R_{K_m}(L,M^2)$. Therefore, by equation \eqref{A_cable} we have
\begin{eqnarray*}
A_{K_m^{(r,2)}}(L,M) &=& (L-1)R_{K_m}(L,M^2)(L+M^{-2r})\\
                &\eqM &\ve(\alpha_{\BJ_{K_m}})(L+M^{-2r})\\
                & \eqM &\ve(\alpha_{K_m^{(r,2)}}).
\end{eqnarray*}
This completes the proof of Theorem \ref{main}.

\end{document}